\documentclass[11pt]{article} 
     \usepackage{amsmath,amsthm} 
     
     \usepackage{amssymb} 
     \usepackage{enumerate} 
\usepackage[all]{xy} 

\usepackage[latin1]{inputenc}
\usepackage[english]{babel}

\newcommand{\A}{\mathbb{A}}

   \newcommand{\PP}{\mathbf{P}} 
   \newcommand{\ZZ}{\mathbb{Z}}

   \newcommand{\CCC}{\mathcal{C}} 
    \newcommand{\Cp}{\tilde{\mathcal{C}}}

    \newtheorem{lemma}{Lemma}[section] 
   \newtheorem{example}[lemma]{Example} 
    \newtheorem{theorem}[lemma]{Theorem} 
    \newtheorem{corollary}[lemma]{Corollary} 
    \newtheorem{proposition}[lemma]{Proposition} 
     
    \newtheorem{assumption}[lemma]{Assumption} 
    \newtheorem{remark}[lemma]{Remark} 
     
    \newtheorem{definition}[lemma]{Definition}

\DeclareMathAlphabet{\mathpzc}{OT1}{pzc}{m}{it} 

\title{Conchoidal transform of two plane curves} 

\date{\empty}

\author{Alberto Albano -- Margherita Roggero \\ 
         Dipartimento di Matematica dell'Universit\`{a} di Torino\\ 
         Via Carlo Alberto 10 \\ 
         10123 Torino, Italy \\ 
         {\small alberto.albano@unito.it} \\ 
         {\small margherita.roggero@unito.it}
       } 

\begin{document}

\footnotetext{Written with the support of the University Ministry funds.} 


\baselineskip=+.5cm

   \maketitle   
\begin{abstract} 
The conchoid of a plane curve $C$ is constructed using a fixed circle $B$ in the affine plane. We generalize the classical definition so that we obtain a conchoid from any pair of curves $B$ and $C$ in the projective plane. We present two definitions, one purely algebraic through resultants and a more geometric one using an incidence correspondence in $\PP^2 \times \PP^2$. We prove, among other things, that the conchoid of a generic curve of fixed degree is irreducible, we determine its singularities and give a formula for its degree and genus. In the final section we return to the classical case: for any given curve $C$ we give a criterion for its conchoid to be irreducible and we give a procedure to determine when a curve is the conchoid of another. 
\end{abstract} 
\section{Introduction} 

The conchoid of a plane curve is a classical construction: given a curve~$C$ in the real affine plane, fix a point~$A$ and a positive real number~$r$. The conchoid of~$C$ is the locus of points~$Q$ that are at distance~$r$ from a point $P\in C$ on the line $AP$. Examples of this construction are the conchoid of Nichomede and the lima\c{c}on of Pascal (see for example \cite{lawrence}, \cite{sendra}). In \cite{sendra} one can find an analysis of the basic algebraic properties of the conchoid of an algebraic plane curve over an algebraically closed field of characteristic zero.

When the curve $C$ is algebraic it is easy to obtain the equation of the conchoid from the equation of $C$. One way to do this is by elimination of variables, using Gr\"obner bases. However, the conchoid of a curve may have multiple components and this procedure does not always give the correct multiplicities. For example, choosing $A = (0,0)$ and~$r = 1$, for the line $x-2=0$ one finds the irreducible quartic $4y^2+x^4+x^2y^2-4x^3-4xy^2+3x^2=0$ while for the line $x=0$ one finds $x(x^2+y^2-1)=0$. In fact in this last case the component $x=0$ should be counted twice.

In this paper we give two different ways to define correctly the conchoid. The first is algebraic, and uses resultants instead of Gr\"obner bases to find the equation of the conchoid. The second is more geometric and uses techniques in algebraic geometry like correspondences and multiple covers of $\PP^2$. 

Our definitions come from an appropriate generalization of the construction of conchoids. First of all, the notion of distance can be replaced with that of intersection with an assigned circle and hence we can work over any field, not only over $\mathbb R$. Moreover, it is more convenient to work in a projective ambient, so for us curve will mean a divisor in $\PP^2$. However the conchoid is essentially an affine concept, and so we fix in $\PP^2$ a line $L_{\infty}$ as line at infinity and a point $A$ in its complement. If $B$ and $C$ are two curves we define the \emph{conchoidal transform of $C$ with respect to $B$} as the locus of points $Q$ intersection of the line $AP$, with $P$ a point in $C$, and the translate of $B$ of a vector $\vec{AP}$. The translation is well defined in the fixed affine part. When $B$ is a circle with center $A$ and radius $r$, this definition is the same as the classical one. In this description the two curves play different roles, but we will see that the conchoidal tranform is in fact symmetrical in $B$ and $C$. 

Both our definitions are universal on the coefficients of the equations of the curves $B$ and $C$. This will allow us to reduce many proofs to the case when one of them is a generic line or a union of generic lines, and use deformations.

After some preliminaries, in section \ref{s:resultant} we give the definition of conchoid using resultants and we prove some properties, in particular we determine the degree, the singularities and the special components of the conchoid. Then in sections \ref{s:incidence} and \ref{s:conchoid} we give a geometric definition. The construction we make works only under suitable hypotheses, which are made explicit in Assumption~$\ref{not1}$ of Section~$\ref{s:incidence}$. We then prove that under these hypotheses the two definitions coincide. 

We use the word \emph{generic} in the sense of algebraic geometry. For a family of objects parametrized by an algebraic variety, for example the family of all plane curves of given degree, generic means ``in the complement of a proper closed algebraic subset'', i.e., outside the locus given by finitely many polynomial equations in the coordinates of the parameter space. Sometimes it is possible to give these equations explicitely, as we do in Assumption~$\ref{not1}$ where every geometric condition can be translated in the vanishing of some polynomials. In other situations it's enough to know that these equations exist, for example in the proof of Theorem~$\ref{main}$ where we use the classical Bertini's theorem.

We show that the conchoid of a generic curve is irreducible and give a formula for its genus that depends on the genera and the degrees of the curves~$B$ and~$C$. The case $B$ circle and $C$ rational is studied extensively in~\cite{sendra2}, where an algorithm is given to determine when the conchoid of~$C$ is rational or splits in two rational components and to compute a rational parametrization of each rational component. We also define the concept of \emph{proper conchoid} in analogy of that of proper transfom. 

In the last section we go back to the classical case: in this situation the multiple cover of $\PP^2$ is a double cover and we use the theory of double planes to give a criterion for the irreducibility of the proper conchoid of any curve. This part requires a bit more algebraic geometry than the rest of the paper, in particular in the proofs we use freely the properties of branched coverings and of normalization. We introduce also the concept of \emph{$n$-iterated conchoid} and show that all the iterated conchoids of a fixed curve belong to a $1$-dimensional flat family. We end with a procedure to determine when an irreducible curve is either the complete or proper conchoid of another.

\section{Notation and generalities}
\label{s:notation}

We work over a fixed base field~$k$. For a geometrical interpretation it is better to have $k$ algebraically closed, but most definitions make sense on the field of definition of the starting curves.  We assume the characteristic of~$k$ to be~$0$ or a prime number~$p$ greater than the degrees of the curves we consider, so we can use derivatives to study singularities.

We will denote $\PP^2$ the projective plane over~$k$. As the concept of conchoid is an affine one, we fix a line~$L_{\infty}$ and we denote with~$\A^2$ its complement. It is a fixed affine plane, and inside it we fix a point~$A$. We choose homogeneous coordinates $[x:y:z]$ in $\PP^2$ so that $L_{\infty}$ has equation $z = 0$, and $A = [0:0:1]$ is the origin of~$\A^2$. If $D\subset \PP^2$ is the curve given by the homogeneous equation $G(x,y,z) = 0$, we denote by $D^{(a)}$ the \emph{affine part} of $D$, i.e., $D^{(a)}$ is the curve in~$\A^2$ given by the equation~$G(x,y,1) = 0$.
	
We fix two projective curves, denoted by~$B$ and~$C$, with equations~$F(x,y,z) = 0$  and $G(x,y,z) = 0$ of degrees~$d$ and~$\delta$ and geometric genus~$g$ and~$\gamma$ respectively. To avoid trivial cases, we assume that $B$ is the projective closure of $B^{(a)}$, i.e., $L_\infty$ is not a component of~$B$. 

\smallskip
The following lemma will allow us to give different but equivalent definitions for the concept of \emph{conchoidal transform of the curves $B$ and $C$}. Having fixed a point~$A$, the affine plane has a natural vector space structure in which $A$~is the zero element. In the statement of the lemma we are using this structure when we add points or multiply them by a scalar.

\begin{lemma}
\label{equiv}  
Let $B$ be a projective curve in $\PP^2$ and $B^{(a)}$ its affine part. For every $P$, $Q$ in $\A^2$, such that $P,\ Q\neq A$, the following are equivalent:
\begin{enumerate}
	\item $Q$ is on the line $AP$ and on the translate of $B^{(a)}$ by the vector $\vec{AP}$;
	\item $Q$ is on the line $AP$ and the point $Q - P$ (i.e., the translate of $Q$ by the vector $\vec{PA}$) belongs to $B^{(a)}$;
	\item $Q = P + S$ with $S\in B^{(a)}$ and $A$, $P$ and $S$ are collinear;
	\item $\exists \ \lambda\in k$ such that $P=\lambda Q$ and $(1-\lambda )Q\in B^{(a)}$.
\end{enumerate}
\end{lemma}

We do not give the proof, which is elementary; we only note that the main reason for the equivalence is the fact that the line $AP$ is invariant under translation by the vector $\vec{AP}$.

\smallskip
Let $C^{(a)}$ be an affine curve. In the classical construction of a conchoid, to each point $P\in C^{(a)}$ one associates the two points on the line $AP$ at distance~$1$ from~$P$. These are the points~$Q$ satisfying condition~\emph{1.} of the previous Lemma, when $B^{(a)}$ is the circle of center~$A$ and radius~$1$. In this way one obtains an affine curve. More generally, one may think of the conchoidal transform of the curve~$C$ with respect to $B$ as the projective closure of the set of points satisfying one of the equivalent conditions of Lemma~\ref{equiv}, as $P$ varies on $C^{(a)}$. Using condition~\emph{3.}, we see that the roles of $C$ and $B$ are in fact completely  symmetrical. We will use condition~\emph{4.} to give a general definition of the conchoidal transform of two projective curves: the definition will involve a resultant, and it will be useful both for theoretical purposes and as a computational device, instead of elimination of variables and Groebner basis computations.

\section{Conchoidal trasforms as resultants}
\label{s:resultant}

Let $B$ and $C$ be projective plane curves, with equations $F(x,y,z) = 0$ and  $G(x,y,z)= 0$ respectively. Writing down explicitely condition~$4$ of Lemma~\ref{equiv} we see that a point $Q = [a:b:1]$ in~$\A^2$ different from~$A$ is in the conchoid of $C$ with respect to $B$ if the system of two equations in the single unknown~$\lambda$
\[ \left\{ \begin{array}{l}
F((1-\lambda) a, (1 - \lambda) b,1) = 0 \\
G(\lambda a, \lambda b,1) = 0
\end{array}\right.
\]
has a solution. Using projective coordinates, we then define:

\begin{definition}
\label{risultante}
The \emph{conchoidal transform} $\CCC(B,C)$ of $B$ and $C$ (which we will often call simply the \emph{conchoid}) is the divisor in $\PP^2$ given by the resultant~$R(F,G)$ of the two polynomials in the homogeneous variables~$\lambda$ and~$\mu$
\begin{equation}
F((\mu-\lambda)x,(\mu-\lambda)y,\mu z) \quad \hbox{ and } \quad G(\lambda x, \lambda y, \mu z)
\end{equation}
\end{definition}

Write $F(x,y,z) = F_d(x,y) + zF_{d -1}+ \dots$ and $G(x,y,z) = G_\delta(x,y) + zG_{\delta -1}+ \dots$   as polynomials in $z$ so that  $F_h$ e $G_h$ are homogeneous polynomials of degree $h$ in the indeterminates $x, y$. We have:
\[
F( (\mu-\lambda) x, (\mu - \lambda) y,\mu z) =  \sum_{i=0}^d \lambda^i \mu^{d-i}\, \Phi_i (x,y,z) 
\]
where $\Phi_i(x,y,z) = (-1)^i \sum_{j=i}^d \dbinom{j}{i} F_j(x,y)z^{d-j}$, and
\[
G(\lambda x, \lambda y,\mu z) = \sum_{i=0}^\delta \lambda^i \mu^{\delta - i}\, G_i(x,y)z^{\delta-i}
\]
hence:
\[R(F,G)=  \begin{array}{|cccccccccc|} 
\Phi_d & \Phi_{d-1}  &\dots  & \dots &\dots &\Phi_0 &0  & \dots &\dots & 0  \\  
 &\dots &  & &\dots & & &\dots & &    \\ 
  &\dots &  & &\dots & & &\dots & &  \\
0 & \dots &\dots & 0& \Phi_d & \Phi_{d-1} & \dots & \dots &\dots &\Phi_0  \\ 
G_{\delta} & z G_{\delta-1}&\dots  & \dots &\dots &  z^\delta G_0 & 0 & 0 & \dots &0 \\
 &\dots &  & &\dots & & &\dots & & \\
  &\dots &  & &\dots & & &\dots & & \\
0 & \dots  &0 &G_{\delta} & z G_{\delta-1}&\dots  & \dots  &\dots&\dots & z^\delta G_0  
\end{array} .\]

\begin{example}\label{32} Conchoidal transform of two lines. Let $F=ax+by+cz$ and $G=a'x+b'y+c'z$. The conchoidal trasform is given by:
\[ 
\left\vert 
\begin{array}{cc} 
-(ax+by) & ax+by+cz \\ 
a'x+b'y & c'z
\end{array}
\right\vert = -\left[(ax+by+cz)(a'x+b'y+c'z)-cc'z^2\right].
\] 
This polynomial does not define a curve only if $B$ and $C$ are both the line $L_{\infty}$ given by $z=0$. In all the other cases it is the hyperbola passing through the origin~$A$ and with asymptotes the lines~$B$ and~$C$. 
\end{example}

\begin{example}
\label{retta2} Conchoid with respect to a line $B$. Let $F=ax+by+cz$ as before and $G$ any homogeneous polynomial of degree~$\delta\geq 2$. The conchoidal trasform is given by:
\[ 
\left\vert 
\begin{array}{ccccc} 
-(ax+by) & ax+by+cz & 0 & \dots & 0\\  
\dots &\dots &\dots &\dots \\ 
0 & \dots & 0& -(ax+by) & ax+by+cz \\
G_{\delta} & z G_{\delta-1}& \dots & z^{\delta - 1}G_1 & z^\delta G_0
\end{array}
\right\vert.
\] 
This polynomial is  
\[
G(x(ax+by+cz),y(ax+by+cz),(ax+by)z)
\] 
as we show by induction on the degree~$\delta$. For $\delta = 1$ it is true by the computation in the previous example; assume now the statement for $\deg G = \delta-1$, and expand the determinant along the last row: we obtain $R(F,G)=(ax+by+cz)^\delta G_\delta(x,y) + z(ax+by)R(F,\overline{G})$ where $G=G_\delta +z\overline{G}$, which is the thesis. 

We have again an effective divisor, unless $B$ and $C$ have both $L_\infty$ as a component.
\end{example}

\smallskip We can obtain the following properties of the conchoidal transform from well known properties of the resultant. In particular, $4.$ and $5.$ say that if either~$B$ or~$C$ is in some special position with respect to~$A$ or~$L_{\infty}$, then the conchoidal transform will have certain special components.

\begin{theorem}
\label{resultant} 
Let $B$ and $C$ be as before. Then:
\begin{enumerate} 
	\item $\deg \CCC(B,C) = 2\delta d$;
	\item $\CCC(B,C) = \CCC(C,B)$;
	\item if $C = C_1 + C_2$ then $\CCC(B,C) = \CCC(B,C_1) + \CCC(B,C_2)$;
	\item if $P \in L_\infty \cap B \cap C$ and  the  multiplicities in $P$ of $B$ and $C$ are respectively $\eta$ and $\epsilon$, then  $P\in \CCC(B,C)$ with multiplicity $\geq \epsilon \delta + \eta d$
	 and  the line $AP$ is a component of $\CCC(B,C)$ with multiplicity $\geq  \epsilon(\eta -\epsilon)+\frac{\epsilon(\epsilon+1)}{2} $ if $\epsilon \leq \eta$;
	\item if $A \in C$ with multiplicity $\nu$, then the divisor $\nu B $ is contained in $\CCC(B,C)$.
\end{enumerate}
\end{theorem}
\begin{proof}
\emph{1.} and \emph{3.} follow respectively from the definition and a property of the resultant (\cite{CLO} Exercise 3 page 79). 

To prove \emph{2.} we observe that the existence of a non trivial solution $(\lambda , \mu)$ is the same as the existence of a non trivial solution $(\lambda'=\mu-\lambda, \mu'=\mu)$ and with respect to these new variables $\lambda', \mu'$ the roles of $C$ and $B$ are interchanged.

To prove  \emph{4.} we may assume that  $P$ has coordinates  $[1:0:0]$ and so the line $AP$ is given by $y=0$. Since the $\eta$ and  $\epsilon$ are the multiplicities in $P$ of $B$ and $C$ respectively, every monomial in $F$ belongs to $(y,z)^\eta$ and every monomial in $G$ belongs to $(y,z)^\epsilon$. Hence every entry in the first $\delta$ rows of the matrix whose determinant is $R(F,G)$ belongs to $(y,z)^\eta$ and every entry in the other $d$ rows belongs to $(y,z)^\epsilon$. This facts clearly imply $R(F,G)\in (y,z)^{\epsilon \delta + \eta d}$.

For the same reason as above, $y^{\eta-i}$ divides $\Phi_{d-i} $ for every $i<\eta$ and $y^{\epsilon -i}$ divides $G_{\delta-i}$ for every $i< \epsilon$. Thanks to \emph{2.} we may assume $\epsilon \leq \eta$. Moreover, after a   change of coordinates $x'=x$, $y'=y$ and $z'=z+ u y$ for a general constant $u$, we may assume that $y^{\epsilon}$ is the maximal power of $y$ dividing $G_{\delta}$ and $y^{\eta}$ is the maximal power of $y$ dividing $\Phi_{d} = F_d$  (note that  such a change of coordinates does not modify $A$ nor $P$). Now the thesis is a consequence of the following Lemma \ref{matrice}, where $S=k[x,y,z]_{(y)}$ and $M$ is the matrix of the first $\eta$~columns of the matrix whose determinant is~$R(F,G)$ (without the null rows).

 Finally, for \emph{5.} we observe that, if $G_0 = \dots = G_{\nu-1} = 0$, then the resultant is multiple of $\Phi_0(x,y,z)^\nu$ where $ \Phi_0(x,y,z)= \sum_{j=0}^d   F_j(x,y)z^{\delta-j} = F(x,y,z)$. 
 \qed \end{proof}

\begin{lemma}\label{matrice} Let $S$ be a discrete valuation ring   with valuation $v$,    let $a,b$ positive integers, $a\leq b$, and  let $M$ be a $2b\times b$ matrix with entries in $S$ of the following type:
\[ 
M = 
\left(
\begin{array}{ccccc} 
 g_1 & g_2  &\dots   &\dots  &g_b  \\  
 0&g_1 & g_2  & \dots     &g_{b-1} \\ 
  &\dots &  &  \\
0 & \dots &0& \dots & g_1  \\ \hline
f_1 & f_2  &\dots   &\dots  &f_b  \\  
 0&f_1 & f_2  & \dots     &f_{b-1} \\ 
  &\dots &  &  \\
0 & \dots &0& \dots & f_1
\end{array} 
\right)
\] 
such that   $v(f_1)=b$, $v(g_1)=a$ and  $v(f_i)\geq b+1-i$ and $v(g_i)\geq a+1-i$ for $i=1,..,b$. Then $v(m)\geq a(b-a)+\frac{a(a+1)}{2}$ for every minor $m$ of maximal order of $M$.
\end{lemma}

\begin{proof}  
As $g_1$ divides $f_1, \dots, f_{b-a+1}$, we can performe a row reduction on $M$ obtaining a matrix of the following type (where we have deleted a final  block of $a$  rows containing only  zero entries): 
\[M'= \left(\begin{array}{ccc|cccc} 
    g_1 & g_2 &\dots  &  &\dots    &\dots  &g_b  \\ 
    & \dots &\dots & &  &   & \dots \\ 
 0&0 & g_1  & & \dots  &     &g_{a+1} \\ \hline 
0 & \dots &0& g_1& \dots  & 0 & g_{a}  \\
 &\dots&   & & \dots  &  \\
 0 & \dots &0& 0& \dots &  0 & g_{1}  \\
0& \dots&  0&  h_{b-a+1}&\dots &\dots   &h_b  \\  
 & \dots &\dots & \dots &    &\dots  &  \\
0&  \dots &0   &  0& \dots &0  & h_{b-a+1} \end{array} \right)\] 
Note that the   statement holds for   $M$ if and only if it holds  for  $M'$. Every $b\times b$ minor of $M'$ 
is the product of the only non-zero $(b-a)\times (b-a)$ minor extracted from the first $b-a$ columns, which is $g_1^{b-a}$  (whose valuation is $a(b-a)$) and an $a\times a $ minor extracted from the lower right block $M''$. All the entries of the $i$-colum of $M''$ have valuation $\geq a+1-i$ and then every $a\times a$ minor of $M''$ has valuation $\geq \frac{a(a+1)}{2}$.
 \qed \end{proof}

\begin{remark} We computed several examples and we have always found that the multiplicity of the line~$AP$ in~$\CCC(B, C)$ satisfies a stronger inequality than the one given in 4. above, namely it is always at least the product $\epsilon \eta$ of the multiplicities at $P$ of the curves $B$ and $C$. We also have some theoretical justification for this fact, but not a proof.
\end{remark}

\section{The incidence surface $W_B=\CCC(B,-)$} 
\label{s:incidence}

The definition just given using resultants is applicable to any pair of curves, gives explicitely the equation of the conchoidal transform and allows to prove some interesting  consequences. However, it is hard in general to obtain geometrical properties from the equation alone. Hence we now present a different characterization of the conchoid of two curves, using a more geometrical approach. In this construction the curves $B$ and $C$ will play different roles and the conchoidal transform will appear as obtained from a fixed curve~$B$ acting over a general curve~$C$.

The definition will use a surface~$W_B$ obtained from the curve~$B$. In this section we define~$W_B$ and study its properties. In the next section we will use it to define the conchoid of~$C$. The geometrical construction  makes sense only if $B$ is generic enough, so we start by fixing the hypotheses on~$B$.

\begin{assumption}
\label{not1} 
$B$ will always be a smooth curve in~$\PP^2$ of degree~$d$ and genus~$g$ (so that $g=1/2 (d-1)(d-2)$), defined  by the equation~$F(x,y,z) = 0$. 

We also assume that $B$ intersects the fixed line~$L_\infty$ in $d$~distinct points~$P_i$,  it does not contain the  fixed point~$A$ and intersects every line through~$A$ in at least $(d - 1)$~distinct points (i.e., no line through~$A$ is a multitangent to~$B$ or a flex tangent). 
\end{assumption}

We will denote by~$L_i$ the $d$~lines~$AP_i$ and by  $D_j$ the  $d(d-1)$~lines through~$A$ that are tangent to~$B$:  we do not exclude that $L_i=D_j$   for some~$i$ and~$j$ may hold. 
	
Finally we will denote by~$B_-$   the curve given by~$F(-x,-y,z)=0$, that is the  curve whose affine part is  symmetric to~$B^{(a)}$  with respect to~$A$.

Let us consider the subset of~$\PP^2 \times \PP^2$ containing all the pairs of points~$(P,Q)$  that satisfy the equivalent conditions given in   Lemma~\ref{equiv} and denote by   $W_B$ its closure (with respect to the Zariski topology). We can write the equations for the affine part of~$W_B$  using condition~$2$ as follows.

Let   $P=[x:y:z]$ and $Q=[X:Y:Z]$ two points not lying on $L_{\infty}$. Then   $(P,Q)\in W_B$ if and only if $xY-yX=0$ and  $F(zX-xZ,zY-yZ,zZ)=0$. The first equation  corresponds to   ``$A$, $P$, $Q$ collinear''  and the second one  to ``$Q-P\in B^{(a)}$''. In fact,  in the affine open set  $z\neq 0$, $Z\neq 0$ the point $Q-P$ is given by  $(\frac{X}{Z}- \frac{x}{z}, \frac{Y}{Z}- \frac{y}{z})$ and so in $\PP^2$ the corresponding point is $[\frac{X}{Z}- \frac{x}{z}: \frac{Y}{Z}- \frac{y}{z}:1]$ that is $[zX-xZ:zY-yZ:zZ]$.

This computation justifies the following definition in projective coordinates: 
\begin{definition}  In the product of  projective planes   $\PP^2 \times \PP^2$ with bihomogeneous coordinates $[x:y:z;\, X:Y:Z]$, the  \emph{incidence surface} with respect to $B$ is the subvariety  $W_B$  defined by the bihomogeneous ideal
\begin{equation}\label{equazW}
I=(F(zX-xZ,zY-yZ,zZ), \ xY-yX).
\end{equation}
We will denote by   $\pi_1: W_B \to \PP^2$ e $\pi_2:W_B \to \PP^2$ the projections on the first and on the second factor.

In a similar way as before, we will use  $W_B^{(a)}$ for the affine part of $W_B$, i.e., $W_B^{(a)}$ is given by the equations $F(X - x, Y - y, 1) = xY - yX = 0$ in  the affine space  $\A^4$  given by  $z\neq 0$, $Z\neq 0$.
 \end{definition}

We note that the subscheme defined by the ideal~$I$ could be reducible or non reduced. However we will prove in the next proposition that under Assumption~$\ref{not1}$ $W_B$ is indeed irreducible and reduced, so it is a variety.
 
\begin{proposition}\label{cilindro} In the above notation:
\begin{enumerate}
	\item  $\pi_1$ e $\pi_2$ are surjective;
	\item  the involution $\sigma $ of $ \PP^2 \times \PP^2$ given by $(P,Q)\mapsto (Q,P)$ restricts to an  isomorphism  $W_B \cong W_{B_-}$. Moreover  $\pi_1 \circ \sigma = \pi_2$, $\pi_2 \circ \sigma = \pi_1$;
	\item  the affine part    $W^{(a)}_B$  of $W_B$ is a product (though in a non-standard way). More precisely: 
\[ W^{(a)}_B\cong B^{(a)}\times \A^1\]
(but $W^{(a)}_B\not \cong \pi_1(W^{(a)}_B)\times \pi_2(W^{(a)}_B)$);
	\item $W_B$ is an irreducible and reduced suface and its affine part  $W^{(a)}_B$ is smooth.
\end{enumerate}
\end{proposition} 

\begin{proof} \emph{1.} By Assumption~\ref{not1} a general line through $A$ in $\A^2$ meets  the affine curve $B^{(a)}$ in $d$~points.  So it is an easy consequence of condition  2. of   Lemma~\ref{equiv} that for a general point~$Q$ on such a line  there is a point   $P$ such that the condition holds for $(P,Q)$ (and viceversa,  for a general  $P$ there is at least a $Q$). Then the image  of $\pi_1$ (or  $\pi_2$) is a dense subset of $\PP^2$. As it is also closed, it must be the whole   $\PP^2$. 

\emph{2.} The isomorphism  between   $W_B$ and  $W_{B_-}$ given by the involution $\sigma$ directly follows from  condition 2. of Lemma \ref{equiv}, because  $Q-P\in B^{(a)}$ if and only if  $P-Q \in B^{(a)}_{-}$.  In the same way we can see that  $\sigma$ exchanges   $\pi_1$ and  $\pi_2$   on the affine subsets. Finally the relations obtained on the affine subset can be extended to the projective closure, because    $\sigma$ is also an involution of   $(\PP^2\times \PP^2) \setminus \A^4$.

 \emph{3.} In the open subset $Z=z=1$, the affine coordinates are $(x,y,X,Y)$.   The equations defining    $W^{(a)}_B$  are $F(X-x,Y-y,1)= xY-yX=0$. With the change of coordinates   $x' = X - x$, $y' = Y - y$ these equations become  $F(x',y',1)=x'y-y'x=0$. Thanks to the hypothesis $A\notin B$, all the solutions can be written as   $(a,b,\lambda a,\lambda b)$  where $[a:b:1]\in B^{(a)}$ (and so $(a,b)\neq (0,0)$). Clearly $W^{(a)} \cong B^{(a)} \times \A^1$.
 
Finally, \emph{4.} is a straightforward consequence of the previous item, because    $W_B$ is the   closure of $W^{(a)}_B$ in $\PP^2\times \PP^2$.
 \qed \end{proof}

We now investigate the singular locus of  $W_B$, that must be contained in  the part at infinity  $W_B\setminus W_B^{(a)}$ because of the previous result. Here we will use the hypothesis that either $\text{char}(k) = 0$  or  $\text{char}(k) = p$  greater than the degree $d$ of $B$.

\begin{proposition} If $d=1$, i.e., $B$ is a line, then $W_B$ is smooth. If $d \geq 2$, the singular locus of $W_B$ is the subvariety $W_B \cap (L_{\infty}\times L_\infty )$ cut by $z=Z=0$. More precisely,  every point in  $W_B \cap (L_{\infty}\times L_\infty )$ has multiplicity at least~$d$.
\end{proposition} 

\begin{proof}     
If  $(P,Q) $ belongs to the locally closed subset of $ W_B$ where  $Z=0$ and $z\neq 0$, then it  has coordinates  of type $[\lambda a: \lambda b: c;a:b:0]$ for some~$\lambda , a, b $ such that $\lambda \neq 0$ and either $a\neq 0$ or $b\neq 0$. If for instance $a\neq 0$ we can choose $a=1$ and consider $(P,Q)$ as a point in the affine $4$-space given by $z=X=1$ and with coordinates $(x,y,Y,Z)$. In these coordinates, the equations of~$W_B$ are
\[
F(1 - xZ, Y - yZ, Z) = xY - y = 0.
\]
We use $F_x$, $F_y$ and~$F_z$ to denote the derivatives of $F(x,y,z)$ with respect to its original variables. Computing the derivatives with respect to the variables~$(x,y,Y,Z)$ using the chain rule and evaluating in the point $(P,Q)=[\lambda : \lambda b: 1;1:b:0]$,  the Jacobian matrix of the equations of $W_B $ is:
  
 \[
\left[ 
\begin{matrix}  
0 & 0& F_y(1,b,0) & F_z(1,b,0)  \\  
b& -1 &\lambda &0
\end{matrix}
\right]
\]
and has rank 2 because $[1,b,0]\in B$ and  $B$ is smooth. We observe that the last item in the first row should be  $-\lambda\left[ F_x(1,b,0)+bF_y(1,b,0)\right]+F_z(1,b,0) $, but the quantity in square brackets vanishes: in fact by the Euler relation it becomes $F(1,b,0)$ and $[1:b:0]\in B$. So we can conclude that $(P,Q)$ is a smooth point. In the same way we can prove the smoothness of every point in the subset of $W_B$ given by  $Z=0$, $z\neq 0$ and $Y\neq 0$. 

The same holds if  $z=0$ and $Z\neq 0$, thanks to the symmetry between $W_B$ and $W_{B_-}$.

\medskip

If $Z=z=0$ then  $(P,Q)=[a:b:0; a: b:0]$ and either $a$ or $b$ does not vanish. If for instance   $x=X=1$, the entries of the first row of the Jacobian matrix  (with coordinates $(y,z,Y,Z)$)   are homogeneous polynomials of degree  $d-1$ with respect to the variables $zX-Zx, zY-Zy, zZ$. If we evaluate the Jacobian matrix  in  $(P,Q)$, that is if we set $y = Y = b$ and $z = Z = 0$, then  its rank is not maximal if and only if  $d\geq 2$. Moreover,  the rank is not maximal also if we consider the higher  derivatives up to the   $(d-1)$-th one. Then  $(P,Q)$ has multiplicity  at least~$d$.
 \qed \end{proof}

We study now the properties of the fibers of the projection~$\pi_1$. We refer to  the beginning of this section for the meaning of $D_j$, $P_i$  and $L_i$.

The fibers of~$\pi_2$ will have the same properties. In fact    $\pi_2$ can be seen as the first projection from the incidence surface $W_{B_-}$. Note that  $B$ and its symmetric curve $B_-$ share the same tangent lines through  $A$ and the same intersection points with the  line at infinity~$L_{\infty}$.

\begin{proposition}
\label{fibre} 
Let $P$ be any point in $\PP^2$.
	\begin{enumerate}
	\item  If  $P$ is general (more precisely if it is not one of the points considered in the following items), then   $\pi_1^{-1}(P)$ is a set of   $d=\deg(B)$ distinct points;
	\item if $P\in D_j\setminus L_{\infty}$, then $\pi_1^{-1}(P)$ is given by   $d-1$ distinct points (exactly one of which with multiplicity 2);
	\item $\pi_1^{-1}(A)$ is the curve $\Gamma$ in $\PP^2\times \PP^2$ of the points $(A,Q)$ such that $F(Q)=0$, so that in a natural way  $\Gamma \cong \pi_2(\Gamma)=B$;
	\item if $P\in L_{\infty}\setminus B$, then $\pi_1^{-1}(P)$ is a single point with multiplicity  $d$;
	\item if $P=P_i[a_i:b_i:0]\in L_{\infty}\cap  B$,  then $\pi_1^{-1}(P_i)$  is the  rational curve $\Lambda_i$ of the points   $[a_i:b_i:0;\lambda a_i:\lambda b_i:\mu]$, so that   $\Lambda_i \cong \pi_2(\Lambda_i)=L_i$.
\end{enumerate}
\end{proposition}

\begin{proof} If  $P$ is a point in $\A^2$, then $\pi_1^{-1}(P)$ can be obtained by first intersecting the line $AP$ with $B$ and then translating by the vector~$\vec{AP}$: this proves~\emph{1.} and \emph{2.} Statement \emph{3.} is the case when  $P = A$ and easily follows from the equations    (\ref{equazW}) of $W_B$. 

So it remains to prove the last two items.   If  $P=[a:b:0]$, looking at the second equation in (\ref{equazW}) we can see that  every point  $Q$ in $\pi_1^{-1}(P)$ is of the type  $Q=[\lambda a:\lambda b:Z]$. If we evaluate the first equation in  $(P,Q)$, we obtain  $F(-aZ,-bZ,0)=0$ that is $(-Z)^dF(a,b,0)=0$. There are two possibilities. If  $P\notin B$, then  $Z=0$ and  $\pi_1^{-1}(P)=\{ [a:b:0;a:b:0]\}$ contains a single point   $Q=[ a: b:0]$ with multiplicity $d$. 
If, on the other hand,   $P\in B$, then all values for  $Z$ are possible and   $\pi_1^{-1}(P)$ is a rational curve with parametric equations $ [a:b:0;\lambda a:\lambda b:\mu]$ in the homogeneous parameters $[\lambda :\mu]$.
 \qed \end{proof}

We  collect in the following corollary the main results obtained until now.

\begin{corollary}
\label{rivest} $W_B$ is a surface in $\PP^2 \times \PP^2$, that is a reduced and irreducible  subvariety of dimension~$2$. If the degree $d = \deg(B) \geq 2$,  its singular locus  is the curve  given by $Z=z=xY-yX=0$ and every singular  point is  $d$-uple. 

 The projection $ \pi_1: W_B \to \PP^2$ is a generically finite map  of degree  $d$,  branched  over the  $d(d-1)$ lines $D_j$, i.e., the lines  containing  $A$ and tangent to $B$. The exceptional fibers are the one over $A$, which is the curve $\Gamma$, isomorphic to $B$  through  $\pi_2$, and  those over  the $d$ points $P_i \in B\cap L_\infty$, which are the rational curves $\Lambda_i$,  isomorphic (through $\pi_2$) to the lines $L_i=AP_i$.
\end{corollary}

\section{The conchoid of $C$ obtained from  $W_B$}
\label{s:conchoid}

If~$C$ is a reduced curve and does not contain any special points   (namely~$A$ and~$P_i\in B\cap L_{\infty}$), then  the curve~$\pi_2(\pi_1^{-1}(C))$ is well defined. Thanks to the equivalent conditions of   Lemma~\ref{equiv}, we can easily see that the curve~$\pi_2(\pi_1^{-1}(C))$   is precisely the conchoidal transform~$\CCC(B,C)$  defined in Section~$3$. However, if either $C$ is non reduced  or it contains some of the special points or some of the special divisors, the curve~$\CCC(B,C)$ can have  some non reduced components and also some components that are in some sense    \emph{special components}. This  is a very common difficulty in  algebraic geometry, when  exceptional fibers of morphisms are involved. Similar to the definition of proper transform for a blowing-up morphism, we   would like  to define a  \emph{proper conchoid}, not containing exceptional fibers of the transformation.  To this end, we give a new definition of conchoid in a  geometric way. We will prove that this definition is equivalent to the previous one, but in it the two starting curves~$B$ and~$C$ play different roles. More explicitly, for every ~$B$ and~$C$ we will obtain not only a curve~$\CCC_B(C)$, but also a set of exceptional divisors: the curve~$\CCC_B(C)$ is the same as $\CCC(B,C)$, but the exceptional divisors will depend only on~$B$, so that they are in general a different set from that of $\CCC_C (B)$.  Removing the exceptional divisors, we will finally obtain the definition of the proper conchoid (Definition~\ref{propria}).

\smallskip
We recall the definitions from algebraic geometry of pull-back and push-forward of cycles. To keep things simple we state them only for the special case we need. For more information on the general notions, see~\cite{fulton}, Chapter~$1$ or \cite{Hart}, Appendix~$A$.

Let $W$ be a projective variety and $\phi : W \to \PP^2$ a surjective morphism. A \emph{$k$-cycle} is a formal linear combination with integer coefficients of reduced and irreducible subvarieties of dimension~$k$. For $C$ a reduced and irreducible subvariety in $W$, we define the push-forward $\phi_*(C)$ to be $0$ if the dimension of the image $\phi(C)$ is strictly less than $\dim C$, otherwise we set
\[
\phi_*(C) = m\cdot \phi(C)
\]
where $m$ is the \emph{degree} of the map~$\phi|_C$, i.e., the number of points in the generic fiber. We extend $\phi_*$ to all cycles by linearity. We note that the push-forward of a $k$-cycle is again a $k$-cycle.

We define the pull-back only for divisors, i.e., cycles of codimension~$1$. If $D$ is a reduced and irreducible curve in $\PP^2$, it is given by a single equation $G = 0$. The pull-back~$\phi^*(D)$ is given by ``pulling back'' this equation to~$W$. To make sense of this we cover $\PP^2$ with open affine sets~$\{U_i\}$, for example the standard ones obtained by dehomogenizing one variable at the time, and we set $g_i$ the corresponding inhomogeneous equation of~$C$ on~$U_i$. Then the $W_i = \phi^{-1}(U_i)$ are an affine open cover of $W$ and we let $\phi_i : W_i \to U_i$ be the restriction of~$\phi$ to~$W_i$. Then $\phi^*(D)$ is the divisor with local equation $g_i\circ\phi_i = 0$ in the open set $W_i$. We note that even if $D$ is reduced, $\phi^*(D)$ may have multiple components since the differential of~$\phi$ may not have maximal rank everywhere. Again we extend $\phi^*$ to all divisors by linearity.

\smallskip
We can now give the:

\begin{definition} If  $C$ is a curve of $\PP^2$, that is a $1$-cycle in $\PP^2$, we will call \emph{conchoid of} $C$ (with respect to  $B$) the cycle $\CCC_B (C)={\pi_2}_* (\pi_1^{*}(C))$. 
 \end{definition}

If $C$ is reduced and does not contain any special point or divisor for $\pi_1$ and $\pi_2$, then $\CCC_B (C)$ is precisely $\pi_2(\pi_1^{-1}(C))$. We can obtain an equation for its affine part   $(\CCC _B(C))^{(a)}$ after  elimination of the variables $x,y$ from the ideal: 
 \begin{equation}
 \label{conc}
 I=(F(X-x,Y-y,1), \ xY-yX, \ G(x,y,1)).
 \end{equation}
In all the other cases, we can consider a flat family of curves   $C_{\mathbf t}$ depending on one or more parameters~${\mathbf t}$  such that  $C_{{\mathbf t}_0}=C$ and, for a general ${\mathbf t}$, $C_{\mathbf t}$ is of the previous type. The conchoid of $C$ is  the limit of $\CCC_B(C_{\mathbf t})$ for ${\mathbf t}={\mathbf t}_0$.
 
 We can for instance consider the family   $C_{\mathbf t}$ of all degree  $\delta $ curves whose equation is a degree $\delta$ polynomial with indeterminate coefficients. Then we can formally performe the elimination of the variables   $x,y$ and,  at the end, specialize ${\mathbf t}$.
 
  It can also be useful to think of the general degree $\delta$ polynomial as an element of a vector space generated by all the products of $\delta$ linear forms, corresponding to curves split in lines.

\begin{example}
\label{es1} 
Let us consider the classical case, when~$B$ is  the circle~$x^2 + y^2 - z^2 = 0$ and $A = [0:0:1]$. The conchoid of a general line $ax + by + cz = 0$, obtained as just indicated, is given by the equation $(aX+bY+cZ)^2(X^2+Y^2) - (aX+bY)^2Z^2 = 0$. If we specialize the coefficients~$a,b,c$ in order to obtain the conchoid of the  line~$L$ of equation~$x = 0$ (which contains $A$), we get $X^2(X^2+Y^2-Z^2) = 0$, i.e., the divisor~$2L+B$ that has degree~$4$. If instead we eliminate, for example using a Gr\"obner basis computation, the variables~$(x,y)$ from the ideal~$I$ in~$(\ref{conc})$ above, the resulting curve has (affine) equation $X(X^2 + Y^2 - 1) = 0$, i.e., it is $L+B$. So we see that specialization does not commute with elimination.
  
  We can also obtain the conchoid of the infinity line $L_\infty$: its equation is $Z^2(X^2+Y^2) = 0$, i.e., the conchoid is $2L_{\infty}+ L_1+ L_2$.
\end{example}

We now state and prove the main result for $B$ and $C$ generic.

\begin{theorem}\label{main} Let $B$ be a curve as in Assumption~\ref{not1}, and let $C$ be a   generic curve of degree~$\delta$ and geometric genus $\gamma$. 
Then:
\begin{enumerate}
\item $\CCC_B(C)=\CCC(B,C)$.
	\item $\CCC_B(C)$ is irreducible and reduced;
	\item $\CCC_B(C)$ is birational to $\pi_1^{-1}(C)$ (via $\pi_2$);
	\item $\CCC_B(C)$ has genus 
	\[ 
	\tilde g =d\gamma +\delta g +(d-1)(\delta -1);
	\]
		\item $\CCC_B(C)$ goes through the origin $A$ with multiplicity~$\delta d$; the tangent cone in the origin is the union of the lines joining~$A$ to the $\delta d$~points of $B \cap C_-$;
	\item $\CCC_B(C)$ meets the line $L_{\infty}$ in the points at infinity of $B$ with multiplicity~$\delta$ and in the points at infinity of $C$ with multiplicity~$d$.
\end{enumerate}
\end{theorem}
\begin{proof} 

Let $f_\delta: \PP^2 \to \PP^N$ be the $\delta$-uple embedding of~$\PP^2$ and $f = f_\delta\circ \pi_1 : W_B \to \PP^N$. Since $W_B$ is an irreducible surface and $\pi_1$ is surjective, we can apply Bertini's Theorem (see \cite[Theorem 3.3.1, pg.~207]{laz})  to obtain that $f^{-1}(H)$ is  reduced and irreducible for a general hyperplane $H \subseteq \PP^N$. By definition of $f_\delta$, a generic plane curve $C$ of degree~$\delta$ is the inverse image of a generic hyperplane of~$\PP^N$ and hence $\pi_1^{-1}(C) = f^{-1}(H)$ is reduced and irreducible for $C$~generic.

As the image of an irreducible subvariety is irreducible we obtain that
\[
\CCC_B (C)={\pi_2}_* (\pi_1^{*}(C)) = mD
\]
where $D$ is a reduced and irreducible plane curve, and $m$ is the degree of ${\pi_2}|_{\pi_1^{-1}(C)}$. We note that $m \le d =$ degree of~$\pi_2 = $ degree of~$B$.

We now show that the cycle $\CCC_B (C)$ has degree~$2d\delta$. The degree is the homology class of the cycle~$\CCC_B (C)={\pi_2}_* (\pi_1^{*}(C))$ in $H^2(\PP^2, \ZZ) \cong \ZZ$, where $\pi_1$ and $\pi_2$ are the restrictions to~$W_B$ of the projections~$p_1$ and $p_2$ defined on~$\PP^2\times \PP^2$. Let $H$ be the class of a line in~$\PP^2$ and let $p$ be the class of a point. The homology module of $\PP^2\times \PP^2$ is free with generators:
\[
\begin{matrix}
A_1 = H\times \PP^2 & \quad  A_2 = \PP^2 \times H\\
a   = p\times \PP^2 & \quad  b   = H\times H & \quad c = \PP^2 \times p\\
\alpha = p\times H & \quad \beta = H \times p\\
\gamma = p \times p\\
\end{matrix}
\]
As homology classes we have:
\[
{\pi_2}_* (\pi_1^{*}(C)) = {p_2}_* ((p_1^{*}(C) \cdot W)
\]
$W$ is a surface, complete intersection of two hypersurfaces of bidegree~$(1,1)$ and~$(d,d)$ and hence its homology class is:
\[
[W] = (A_1 + A_2)\cdot(dA_1 + dA_2) = d(A_1 + A_2)^2 = d(a + 2b + c)
\]
$C$ is a plane curve of degree~$\delta$ and hence $[C] = \delta H$. Then $p_1^{*}(C) = \delta H \times \PP^2 = \delta A$. Intersecting with $W$ we get 
\[
p_1^{*}(C) \cdot W = d \delta A\cdot (a + 2b + c) = d \delta (2\alpha + \beta)
\]
We have ${p_2}_*(\beta) = 0$  since the image of $\beta$ is a point, while ${p_2}_*(\alpha) = H$. We conclude
\[
[{\pi_2}_* (\pi_1^{*}(C))] = {p_2}_* ((p_1^{*}(C) \cdot W) = 2d\delta H \in H^2(\PP^2, \ZZ).
\]

\smallskip
We now let $B = L$ a line not through the point~$A$ and let $C$ be reduced and irreducible, not through~$A$ and such that $B \cap C \cap L_\infty = \emptyset$. Since $m \le d = 1$, we have that
\[
\CCC_L(C) = D
\]
is reduced and irreducible and has the same degree as $\CCC(L, C)$. As the affine parts  of these two curves have the same support, we conclude that they are equal as cycles. In particular, $\CCC(L, C)$ is reduced and irreducible for $C$ generic.

Since the definition via resultants is symmetric, taking $B$ as in Assumption~$\ref{not1}$ and $L$ a generic line, we have that $\CCC(B, L)$ is reduced and irreducible, and by Theorem~\ref{resultant} $\CCC(B, C)$ is reduced when $C$ is the union of $\delta$~distinct generic lines. The property of being reduced is an open property, since one can detect multiple roots with the vanishing of system of resultants (the discriminants), and hence we have that $\CCC(B, C)$ is reduced for $C$~generic.

Consider again $\CCC_B(C)$ and $\CCC(B, C)$ for $C$ generic: their affine parts have the same support, $\CCC(B, C)$ is reduced, $\CCC_B(C) = mD$ has only one irreducible component, and they have the same degree as cycles. We conclude that they are equal as cycles, and hence they are both reduced and irreducible. Moreover, $m = 1$ and so ${\pi_2}|_{\pi_1^{-1}(C)}$ is a birational map from $\pi_1^{-1}(C)$ to $\CCC_B(C)$. This proves \emph{1.}, \emph{2.}, and \emph{3.}

\smallskip
\emph{4.} By what we have just proved, it is enough to compute the genus of~$\pi_1^{-1}(C) = \tilde C$. The map $\pi_1 : \tilde C \to C$ is a covering of degree~$d$, ramified over the points where~$C$ meets the ramification of~$\pi_1$, i.e., the~$d(d-1)$ lines through~$A$ tangent to~$B$. By our assumption on~$B$ the ramification index is~$1$ for all these points, and so the Riemann-Hurwitz formula gives:
\[
2\tilde g - 2 = d(2\gamma - 2) + \delta d (d-1);
\]
since $B$ is smooth of degree~$d$, its genus~$g$ equals $\dfrac{(d-1)(d-2)}{2}$, from which the thesis follows. We write the genus formula in this way to point out once again the symmetry between $C$ and $B$.

\smallskip
\emph{5.} and \emph{6.} now follow from what we have proved, and the fact that they are true when $C$ is a generic line (see Example \ref{retta2}). In fact, if $C$ is the line of equation $ax + by + cz = 0$, $\CCC_B(C)$ has equation
\[
F(x(ax + by + cz), y(ax + by + cz), (ax + by)z) = 0
\]
Setting $z = 1$, the homogeneous part of minimum degree is $F(cx, cy, ax + by)$ and setting this to zero gives the tangent cone in the origin. Hence we see that the tangent cone is given by the lines through the origin and the points of intersection of $B$ and the line $ax + by - cz = 0$, which is $C_-$.
\qed 
\end{proof}

\begin{corollary}  Let $B$ be a curve as in Assumption~\ref{not1}. Then for every curve $C$ we have $\CCC(B,C)=\CCC_B(C)$. 
Moreover 5. and 6. of Theorem~\ref{main} still hold, with the multiplicities greater than or equal to the ones given (instead of just equal).
\end{corollary}

\begin{proof}

Denote with $G_\mathbf{t}(x,y,z)$ the generic form of degree~$\delta$ in the indeterminates~$x,y,z$, denote with $\mathbf{t}$ its coefficients which we take as indeterminates, and let $C_{\mathbf t}$ be the corresponding curve. Using the definition via resultants, we can determine~$\CCC(B,C_{\mathbf t})$: as a function of the variables~$\mathbf{t}$ it is given by a polynomial~$R_{\mathbf{t}}$.

Also $\CCC_B(C_{\mathbf t})$ is given by a polynomial function of the variables~$\mathbf{t}$, and the two polynomials must coincide up to a constant factor since the curves obtained by generic specialization of~$\mathbf t$ coincide thanks to Theorem~\ref{main}. 
Hence the two curves~$\CCC(B,C)$ and $\CCC_B(C)$ coincide for all choices of~$C$.

Since also \emph{5.} and \emph{6.} are given by properties of the polynomials defining $\CCC(B,C_{\mathbf t})$ and $\CCC_B(C_{\mathbf t})$, the same reasoning shows that they hold in general, as inequalities, by semicontinuity in~$\mathbf t$ .
 \qed \end{proof}

We note that if the surface~$W_B$ is reducible our proof does not work since we cannot use Bertini's theorem. Our Assumption~$\ref{not1}$ is used not only to ensure the irreducibility of $W_B$, but also to obtain the genus formula given in the previuos theorem.

We want to emphasize that specialization of parameters does not commute with elimination of variables, as we have seen in Example~$\ref{es1}$, while it commutes with the resultants. Our geometric definition coincides with the one given by resultants, and so it commutes with specialization.

\smallskip
Recall the definition and the properties of the divisors~$\Gamma$ and~$\Lambda_i$ in $W_B$ that we will consider as \emph{special}.  $\Gamma $ is  the divisor $ \pi_1^{-1}(A)$ and $\Lambda_i = \pi_1^{-1}(P_i)$ where $P_i \in B\cap L_\infty$. We have   $\pi_2(\pi_1^{-1}(A))=\pi_2(\Gamma) = B$ and  $ \pi_2(\Lambda_i) =L_i$.

\begin{definition}
\label{propria} 
Let $C$ be a curve. The \emph{proper conchoid of $C$ with respect to $B$} is the curve $\Cp_B(C)$   that does not have $B$ and the $L_i$'s as components and such that $\CCC_B (C)=aB + \sum_i b_i L_i + \Cp_B(C)$. 
\end{definition}

By what has been proved, the integers~$a,\ b_i$ are greater than or equal to the multiplicities of $C$ in the points~$A,\ P_i$; they are strictly greater if the tangent cone to $C$ in one of these points contains one of the lines~$L_i$.

\begin{remark} This definition has one drawback: it always eliminates the curve~$B$ from the conchoid of another curve, even when $B$ should be considered as a non-exceptional component. For example, if $B$ is the circle with center~$A$ and radius~$1$ and $C$ is the circle with the same center~$A$ and radius~$2$, $B$ should be considered a non-exceptional component of the conchoid of~$C$ with respect to~$B$, since $C$ does not go through~$A$.  
\end{remark}

\section{The classical case: $W_B$ and  double planes}
\label{s:doubleplane}

We want now to apply our results in the classical case, i.e., when $B$ is a circle with center~$A$. In this case we show that $W_B$ is the blow-up in three points of a ramified double cover of $\PP^2$. The geometry of these surfaces, classically known as \emph{double planes}, is well-known and this will allow us to determine sufficient conditions on the curve~$C$ so that its conchoid is irreducible. 

The same approach could be followed for curves~$B$ of any degree~$d$. $W_B$ is again a blow-up of a ramified cover of $\PP^2$ of degree~$d$, but in this case the geometry of multiple covers is much less known, and little can be said in general.

For  a clear exposition in modern language of  the classical theory  of double planes see for instance the paper by Sernesi \cite{Sernesi}. In particular, in that paper one can find necessary and sometimes sufficient conditions on a curve in $\PP^2$ so that its pullback to the double cover is reducible. Stated loosely, the condition is that the curve must be everywhere tangent to the branch locus. We do not use directly this, since it requires that the branch locus is smooth and the curve generic and in our case the branch locus is a pair of lines. However, the statement turns out to be true for the particular double plane we are interested in and valid for all irreducible curves as we will prove.

\medskip
Let $B$ be a circle with center~$A$ or, more generally, a conic with center in~$A$. There are two points~$P_1$ and~$P_2$ in $B\cap L_{\infty}$ and hence two lines~$L_1$ and~$L_2$. These lines are also the tangents to~$B$ passing through~$A$, previously denoted~$D_i$, since the center of the conic is the pole of the line at infinity.  

Let $D $ be  the  cycle $ L_1 + L_2 + 2L_{\infty}$ in~$\PP^2$, i.e., the curve (reducible and not reduced) with equation~$\ell_1\ell_2 z^2 = 0$,  $E$ the double plane branched over~$D$ and $p: E \to \PP^2$  the corresponding finite morphism of degree~$2$.

A point on the surface $E$ is singular if and only if its image under~$p$ is a singular point of~$D$, and in this case it is a double point on~$E$. Since $D$ has a multiple component, $E$ is not normal and has a curve of double points that projects onto~$L_\infty$. Moreover, $\tilde A = p^{-1}(A)$ is an ordinary double point of $E$.

Let $n : F \to E$ be the normalization morphism (see, e.g., \cite[I, Exercise 3.17]{Hart}): the composition $q = p\circ n : F\to \PP^2$ is a double plane branched over the divisor $L_1 + L_2$, and hence $F$ is a quadric cone. For a proof of all of the preceding assertions, see~\cite{Sernesi}, especially Teorema page~$19$ and Esempio page~$8$ and~$9$.

It follows that $E$ is obtained from a quadric cone by identifying two rational curves (that are not lines on the cone, since they project onto the line~$z=0$). In particular, we obtain an isomorphism between the open sets $\pi_1^{-1}(\A^2\setminus \{0\})$ of $W_B$ and $q^{-1}(\A^2\setminus \{0\})$  of $F$.

We summarize the construction in the following diagram:

\begin{equation}
\label{diagram}
\xymatrix{
F \ar[r]^{n} \ar[drr]_{q} &          E  \ar[dr]^{p}  &        &  W_B \ar[ll]_{f} \ar[dl]^{\pi_1}  \ar[dr]_{\pi_2}\\
   &                          & \PP^2  &     & \PP^2
}
\end{equation}
where $f$ is the blow-up of $E$ in $\tilde A$ and the two points over $P_1$ and $P_2$, the points at infinity of $L_1$ and~$L_2$, as can be seen from the description of the geometry of~$W_B$ given in Proposition~\ref{fibre}.

So the proper conchoid $\Cp_B(C)$ of Definition~\ref{propria} is birational to the corresponding proper transfom in~$F$. We already proved that a generic irreducible curve has irreducible conchoid (and hence irreducible proper conchoid). Using this description via double planes we can now characterize completely the curves whose proper conchoid is irreducible.
 
 \begin{theorem}\label{t:spezzata} Let $C$ be a reduced and irreducible curve in $\PP^2$ of degree~$\delta$ with equation $G(x, y, z) = 0$. Then $\Cp_B(C)$ is reducible if and only if there exist homogeneous polynomials~$H_1$ and~$H_2$ such that:
\begin{enumerate}
	\item $G = H_1^2-\ell_1\ell_2H_2^2$\qquad   if $\delta$ is even 
	
	or
	
	\item $G = \ell_1 H_1^2 - \ell_2 H_2^2$\qquad  if $\delta$ is odd
\end{enumerate}
where $\ell_1$ and $\ell_2$ are the equations of the lines $L_1$ and $L_2$ respectively.
\end{theorem}

\begin{proof}
By Diagram~$(\ref{diagram})$ $\Cp_B(C)$ is irreducible if and only if $q^{-1}(C)$ is. So it is enough to prove the claim for $q^{-1}(C)$, and even restrict ourselves to the affine case. Identifying $F$ with the quadric cone of equation~$\ell_1\ell_2 - t^2$ in~$\mathbb{A}^3$, the map~$q$ is the projection onto the plane~$\A^2$ given by~$t=0$. The quadric cone is normal and the degree map is an isomorphism from its divisor class group to~$\ZZ/2\ZZ$. In particular a divisor is a complete intersection if and only if it has even degree  (see \cite{Hart}, Ch. I, Exercise 3.17 and Ch. II, Example 6.5.2).
 
If $C$ has equation~$G(x, y) = 0$, then $q^{-1}(C)$ is given by the intersection of the cone with the hypersurface in~$\A^3$ of equation~$G(x,y) = 0$. If this divisor on the cone is reducible then it has exactly two components, since $q$~is a finite map of degree two. If $Y_1$ is one of the components, the other component $Y_2$ is obtained using the involution $t \mapsto -t$ on the cone. Hence both components have degree~$\delta$.

If $\delta$ is even, then $Y_1$ is a complete intersection of the cone with a hypersurface~$H(x,y,t)$ of degree~$\delta/2$, and since we are taking intersection with the cone~$t^2 = \ell_1\ell_2$, we can assume that $H = H_1 + tH_2$, where  $H_1, H_2 \in k[x,y]$. In this case, $H_1 - tH_2$ gives the other component~$Y_2$ and $H_1 ^2 - t^2H_2^2$ cuts on the cone the sum $Y_1 + Y_2 = q^{-1}(C)$.   Using again the equation $t^2 = \ell_1\ell_2$ of the cone, we see that $H_1 ^2 - \ell_1\ell_2H_2^2$ cuts the same divisor $Y_1 +Y_2$. Finally, $H_1^2 - \ell_1\ell_2 H_2^2 = 0$ as a curve in $\A^2$ contains the curve~$C$ and has the same degree, and hence coincide with~$C$. So the equation $G(x,y)$ of $C$ is as claimed.

The proof is similar in the case $\delta$ odd. The divisor~$Y_1$ is not principal, and we consider the principal divisor $Y_1 + L$, where $L$ is the line $\ell_1 = t = 0$. As before, $Y_1 + L$ is cut on the cone by a hypersurface that in this case will be of the form~$\ell_1 H_1 + tH_2$ with $H_1,\ H_2 \in k[x,y]$ because it must vanish on $L$. Then $Y_2 + L$ is cut by  $\ell_1H_1 - tH_2$, $Y_1 + Y_2 + 2L$ is cut by $\ell_1^2H_1^2 - \ell_1\ell_2H_2^2$ and finally $Y_1 + Y_2$ by $\ell_1H_1^2 - \ell_2H_2^2$ as $\ell_1$ cuts $2L$ on the cone.
 \qed \end{proof}

\begin{example} Let $B$ be the circle with center~$A$ and any radius. By the previous theorem, a conic~$C$ has reducible conchoid if and only if $C$ has equation $\ell^2 -\ell_1 \ell_2 = 0$, where $\ell$ has degree~$1$. Notice that ~$\ell$ is the equation of the polar line of~$A$ with respect to~$C$. Since  $B$ is a circle, the points $P_1$ and $P_2$ are the cyclic points. Then $\ell_1$ and $\ell_2$ are two tangents to the conic~$C$ from the cyclic points and so their intersection~$A$ is a focus of~$C$ (by definition of focus, see, e.g., \cite[pag.~171]{coolidge}). Note that a complex conic which is not a parabola has $4$~foci, and if the conic is real only two of them are real. These are the usual foci of ellipses and hyperbolas. Since a parabola is tangent to the line at infinity, it has always only~$1$ focus.
We conclude that a conic has reducible conchoid if and only if $A$~is one of the foci of~$C$. For instance if $B$ is the circle with center~$A$ and   radius $1$ and $C$ is the parabola $(y+z)^2-(x^2+y^2)$, then  $\Cp_B(C)$ is the union of the two quartics $x^4+(y^2-2yz)x^2-2y^3z+y^2z^2=0$ and $x^4+(y^2-2yz-4z^2)x^2-2y^3z-3y^2z^2=0$.

   For a different proof and many explicit computations, see~\cite[Theorem 6]{sendra2}.
\end{example}

\smallskip
We assume now that $C$ is reducible. We can again ask if its proper conchoid is reducible or not. To answer this question we introduce the notion of \emph{iterated conchoid}. We begin with an example.

\begin{example}
\label{riconcoidi} Let $C$ be a generic line. Let $C_1 = \CCC_B(C)$ be its conchoid, which is again irreducible, and let us consider $\CCC_B(C_1)$. This is a divisor of degree~$16$, whose components are the circle~$B$ with multiplicity~$2$, the two lines~$L_1$ and~$L_2$ each with multiplicity~$3$, the line~$C$ with multiplicity~$2$ and an irreducible curve~$C_2$ of degree~$4$ (to check this computation, take $C$ the line of equation $x - hz = 0$ and use resultants). The curve $C_2$ is in fact the conchoid of $C$ with respect to the circle~$B_2$ with center~$A$ and radius twice that of $B$. 
\end{example}

This behaviour is not special to the lines and we prove:

\begin{proposition} \label{p:riepilogo}
Let $C$ be a generic curve of degree~$\delta$, and let $C_1=\CCC_B(C)$. Then the conchoid~$\overline{C_2}=\CCC_B(C_1)$  is a divisor of degree~$16\delta$, whose components are:
\begin{enumerate}
	\item the circle $B$, with multiplicity $2\delta$;
	\item the two lines~$L_1$ and~$L_2$ each with multiplicity~$3\delta$;
	\item the curve~$C$ with multiplicity~$2$;
	\item a curve $C_2$ of degree~$4\delta$, which is the conchoid of~$C$ with respect to the circle~$B_2$ with center~$A$ and radius twice that of $B$. 
	\end{enumerate}
\end{proposition}

\begin{proof} As we did earlier, we can consider $C$ as specialization of the curve~$C_{\mathbf t}$ of degree~$\delta$ with generic coefficients. The linear system $C_{\mathbf t}$ is generated by curves that are product of generic linear forms. Hence every curve $\CCC_B(\CCC_B(C_{\mathbf t}))$, and so especially $\overline{C_2}:=\CCC_B(\CCC_B(C))$, must contain $B$, $L_1$ and $L_2$ with at least multiplicity as stated. 

Now, let us consider  the affine part~$\overline{C_2}^{(a)}$ of~$\overline{C_2}$ . It contains    all the points~$Q$ of the form $Q = P' + (P+S)$, where  $S\in C^{(a)}$ and $P',P \in B^{(a)}$ collinear with~$A$ (see Lemma~\ref{equiv}). The intersection~$B\cap AS$ consists of two points~$P_+$ and~$P_-$ and hence there are two possibilities for~$P$ and two for~$P'$. If either $P = P' = P_+$ or $P = P' = P_-$, the corresponding point~$Q$ belongs to~$B_2$ and hence $Q$ belongs to $\CCC_{B_2}(C)$.   Since the total degree is~$16 \delta$ the components appear with the stated multiplicity, and not higher, and their sum is  the whole divisor $\CCC_B(C_1)$.
 \qed \end{proof}

\begin{definition} 
The curve  $C_2$ defined in the previous Proposition is called \emph{proper second conchoid} of~$C$. 
\end{definition}

In this case we discard from~$\CCC_B(C_1)$ not only the exceptional components, but also the curve~$C$.

\begin{remark} We can define inductively the proper $n$-th conchoid $C_n$ of $C$ and see in the same way that it  turns out to  be the conchoid of $C$ with respect to the circle $B_n$ with center $A$ and radius $n$ times that of $B$. The infinitely many curves $C_n$  belong to a $1$-dimensional flat family.  In fact $C_n=\CCC_{B_n}(C)=\CCC(B_n,C)$   can be obtained using the resultant  $R(F_t,G)$, where $F_t=x^2+y^2-t^2z^2$,  and   specializing the parameter $t$ to $n$.   
\end{remark}

\begin{proposition} 
The proper conchoid of a reducible curve is reducible.
\end{proposition}
\begin{proof}
Let $C$ be a reducible curve and $\Delta:=\pi_1^{-1}(C) \setminus \{\text{exceptional components}\}$. $C$ is reducible and so $\Delta$ has at least two components, since $\pi_1(\Delta) = C$ has a number of components less than or equal to that of~$\Delta$. Assume that the proper conchoid $\Cp(C)$ is irreducible. As the map~$\pi_2$ is generically $2:1$,  $\pi_2^{-1}(\Cp(C))$ has at most two components and since it contains~$\Delta$ it must be $\pi_2^{-1}(\Cp(C))=\Delta$ and hence $\pi_1(\pi_2^{-1}(\Cp(C))) = \pi_1(\Delta) = C$. Since $B=B_-$, the curve~$\pi_1(\pi_2^{-1}(\Cp(C)))$ contains the proper second conchoid of~$C$ and hence it has at least~3 irreducible non exceptional components, those of~$C$ and~$\CCC_{B_2}(C)$. This contradiction proves our claim, since a component of $C$ cannot be equal to~$\CCC_{B_2}(C)$: in fact any curve different from~$L_\infty$ or an exceptional curve has only finitely many points in common with its conchoid.
 \qed \end{proof}

\medskip
We conclude giving a computational procedure to establish when an irreducible curve~$\mathcal D$ is either the  conchoid or the proper conchoid of another curve~$C$ with respect to some point~$A$ (not necessarily the origin) and radius~$r$, i.e., with respect to the circle $B$ with equation~$(x-a)^2 + (y-b)^2 - r^2z^2 = 0$. 

In order to decide if~$\mathcal D$  is a complete conchoid,  we start by checking some obvious necessary conditions: first of all the degree must be a multiple of~$4$. If we set $\deg(\mathcal D) = 4\delta$, then $\mathcal D$ must meet the line at infinity~$z=0$ in the two cyclic points ($[1:i:0]$ and $[1:-i:0]$) with multiplicity at least~$\delta$ and all the other points at infinity of~$\mathcal D$ must be at least double points. Hence, if $H(x,y,z)=0$ is an equation defining~$\mathcal D$, then $H(x,y,0)$ must split as $(x^2+y^2)^\delta H_\delta(x,y)^2$. Moreover, there must be a point on~$\mathcal D$ (namely the point $A$) in the affine open set~$\A^2$ with multiplicity at least~$2\delta$. 

When all these conditions are fulfilled, the distance~$r$ must be twice the distance between a pair of points on~$\mathcal D$ and collinear with~$A$.

Hence the only possibilities for $A$ and $r$ are finite, and we can check all cases to see if the conchoid of~$\mathcal D$ with respect to the circle with center~$A$ and radius~$r$  contains a non-exceptional component with multiplicity~$2$: for what we proved above this component, if it exists, is a curve whose conchoid is~$\mathcal D$.

\medskip

In order to check if~$\mathcal D$ is the proper conchoid of a curve $C$ we can use Theorem~\ref{t:spezzata} and Proposition~\ref{p:riepilogo}. Excluding the trivial case $\deg(\mathcal D) = 1$, a first necessary condition is the existence of the pair of lines $\ell_1$, $\ell_2$,  each containing  a cyclic point, which are everywhere tangent to~$\mathcal D$. If they exist and they meet in the affine subset~$\A^2$, their common point  is $A$ and, as above,  the distance~$r$ must be twice the distance between a pair of points on~$\mathcal D$ and collinear with~$A$.   Hence there are finitely many possibilities for  $r$ and we can check all cases to see if the conchoid of~$\mathcal D$ with respect to the circle with center~$A$ and radius~$r$ splits as described in Proposition~\ref{p:riepilogo}: if ~$\mathcal D=\CCC_B(C)$, the curve~$C$  is a non-exceptional component with multiplicity~$2$ of~$\CCC_B({\mathcal D})$.


\begin{thebibliography}{12}
   
   \bibitem{coolidge} Coolidge, J. L.: A Treatise on Algebraic Plane Curves, Dover Publ. Inc., New York, 1959
    
   \bibitem{CLO}     Cox, D., Little, J., O'Shea, D.: Using Algebraic Geometry, 
Graduate Texts in Mathematics, Volume 185, Springer, 2005

   \bibitem{fulton} Fulton, W.: Intersection Theory, Springer-Verlag, New York (1984)
   
  \bibitem{Hart}   Hartshorne, R.: Algebraic Geometry, Graduate Texts in Mathematics, Volume 52, Springer-Verlag, New York (1977) 

  
  \bibitem{lawrence} Lawrence, J. D.: A catalog of special plane curve, Dover Publ. Inc., New York, 1972
  
  \bibitem{laz} Lazarsfeld, R.: Positivity in Algebraic Geometry I, Springer-Verlag, (2004)
  
  \bibitem{sendra} Sendra, J. R., Sendra, J.:
 \emph{An algebraic analysis of conchoids to algebraic curves}, Appl. Algebra Eng., Commun. Comput.,
\textbf{19}, {5}, {2008}

  \bibitem{sendra2} Sendra, J. R., Sendra, J.: \emph{Rational Parametrization of Conchoids
to Algebraic Curves}, Appl. Algebra Eng., Commun. Comput., this issue

  
   \bibitem{Sernesi} Sernesi, E.:  \emph{Introduzione ai piani doppi}, Seminario di Geometria 1977-1978, Centro di Analisi Globale, C. N. R.
   \end{thebibliography}
\end{document}